\documentclass[12pt,letterpaper,reqno]{amsart}
\usepackage[letterpaper,margin=1.2in,headheight=15pt]{geometry} 
\usepackage{graphicx,bbm}
\usepackage{amsmath, amssymb, amsfonts,amsthm,mathrsfs}
\usepackage[utf8]{inputenc}
\usepackage{epstopdf}
\usepackage{tikz-cd} 
\usepackage{setspace}
\usepackage{xcolor}
\usepackage{tcolorbox}
\usepackage[titletoc,title]{appendix}
\usepackage{enumitem}
\usepackage{mathtools}
\usepackage{newtxmath} 
\usepackage{hyperref}
\usepackage{tikz}
\usepackage{float}
\usetikzlibrary{arrows.meta, decorations.pathmorphing, calc}
\definecolor{darkred}{rgb}{0.5,0.15,0.15}
\hypersetup{colorlinks=true,urlcolor=darkred,linkcolor=darkred,citecolor=darkred}

\newcommand{\be}{\begin{eqnarray}}
\newcommand{\ee}{\end{eqnarray}}
\newcommand{\bea}{\begin{eqnarray}}
\newcommand{\eea}{\end{eqnarray}}

\newcommand{\ben}{\begin{eqnarray}}
\newcommand{\een}{\end{eqnarray}}

\theoremstyle{plain}
\newtheorem{thm}{Theorem}[section]
\newtheorem{prop}[thm]{Proposition}

\theoremstyle{definition}

\newtheorem*{question*}{Question}

\theoremstyle{remark}
\newtheorem{rem}[thm]{Remark}

\numberwithin{equation}{section}

\title{Modularity of Feynman Integrals and Factorization of Appell F2 Systems}
\author{Murad Alim}
\author{Filippo La Mantia}
\address{Maxwell Institute for Mathematical Sciences, Edinburgh EH14 4AS, UK\\
Department of Mathematics, Heriot-Watt University, Edinburgh EH14 4AP, UK\\
\small Department of Mathematics, Technical University of Munich, Boltzmannstr. 3, 85748, Garching, Germany}

\begin{document}

\maketitle

\begin{abstract}
Certain Feynman integrals can be expressed as periods of differential forms on Calabi--Yau manifolds. We provide a mathematical proof of a result of Duhr and Maggio on the modularity of the two-dimensional conformal traintrack integral. Our approach is based on a factorization of the associated Picard-Fuchs system into a tensor product of Gauss hypergeometric systems via a gauge transformation due to Clingher, Doran and Malmendier.
\end{abstract}
\vspace{-0.1cm}
\section{Introduction}
Some families of Feynman integrals admit a natural interpretation as periods of algebraic varieties and their differential equations encode geometric information, such as variations of polarized Hodge structures and modularity, see e.g. \cite{wein,Doran2,Duhr2}. In the case of the \emph{fishnet integrals} \cite{fish1,fish2,fish3}, the associated varieties are families of Calabi-Yau manifolds. In this note we investigate a particular example in this class, namely the conformal two-loop traintrack integral considered by Duhr and Maggio \cite{duhr2025feynmanintegralsellipticintegrals}. The associated geometry is a two-parameter family of K3 surfaces, whose Picard-Fuchs system is given by the Appell $F_2$ hypergeometric system. Our main result is a factorization of this system into the tensor product of two single-variable Gauss hypergeometric systems, obtained via a gauge transformation due to Clingher, Doran and Malmendier \cite{Clingher_2017}. This provides a direct proof of the modular parametrization of \cite{duhr2025feynmanintegralsellipticintegrals}, previously obtained through an ansatz and matching terms of the power series expansions, yielding a basis of periods that is well suited for analyzing the modularity of the family. Geometrically, the factorization reflects the Kummer surface structure of the underlying K3 which can be constructed out of a product of elliptic curves \cite{Clingher_2017,Griffin_2018}. This example fits into the broader framework in which Picard–Fuchs systems of algebraic varieties, particularly for elliptic and K3 families, admit factorization structures and modular parametrizations, see e.g. \cite{Doran2000, LianYau1996I,LianYau1996II,Alim:2018vkq}.
\section*{Acknowledgements}
We would like to thank A. Clingher, C. Doran, C. Duhr, A. Malmendier for comments, discussions and correspondence.

\section{The hypergeometric functions $_2F_1$ and $F_2$}
We briefly recall the hypergeometric functions relevant for our analysis. The Gauss hypergeometric function ${}_2F_1$ is characterized as a solution of the differential equation:
\begin{equation} \label{def:2F1sys} \small
z(1-z)\frac{d^2 f}{dz^2} + \bigl(\gamma - (\alpha+\beta+1)z\bigr)\frac{df}{dz} - \alpha\beta f = 0.
\end{equation}
For $\operatorname{Re}(\gamma)>\operatorname{Re}(\beta)>0$, it admits the Euler integral representation:
{\small \[
{}_2F_1\!\left(
\begin{array}{c}
\alpha,\beta\\
\gamma
\end{array}
\middle|\, z
\right) =\frac{\Gamma(\gamma)}{\Gamma(\beta) \Gamma(\gamma-\beta)} \int_{0}^{1} \frac{d t}{t^{1-\beta}(1-t)^{1+\beta-\gamma}(1-z t)^{\alpha}}. \label{def:2F1int}
\] }

A two-variable generalization is the Appell $F_2$ function, which is defined as a solution of the differential system:
\begin{equation} \label{def:F2sys} \small
\begin{cases}
     x\left(1-x\right) \frac{\partial^{2} F}{\partial x^{2}}-x y \frac{\partial^{2} F}{\partial x \partial y}+\left(\gamma_{1}-\left(\alpha+\beta_{1}+1\right) x\right) \frac{\partial F}{\partial x}-\beta_{1} y \frac{\partial F}{\partial y}-\alpha \beta_{1} F=0  \\ \\
 y\left(1-y\right) \frac{\partial^{2} F}{\partial y^{2}}-x y \frac{\partial^{2} F}{\partial x \partial y}+\left(\gamma_{2}-\left(\alpha+\beta_{2}+1\right) y\right) \frac{\partial F}{\partial y}-\beta_{2} x \frac{\partial F}{\partial x}-\alpha \beta_{2} F=0 
\end{cases}
\end{equation}
For $\operatorname{Re}\left(\gamma_{1}\right)>\operatorname{Re}\left(\beta_{1}\right)>0$ and $\operatorname{Re}\left(\gamma_{2}\right)>\operatorname{Re}\left(\beta_{2}\right)>0$, the function admits the following integral representation, as proven in \cite{Clingher_2017}:
{\small \begin{align}  
F_2\!\left(
\begin{array}{c} 
\alpha;\,\beta_1,\beta_2\\
\gamma_1,\gamma_2
\end{array}
\middle|\, x,y
\right)
&=
\frac{\Gamma(\gamma_1)\Gamma(\gamma_2)}
{\Gamma(\beta_1)\Gamma(\gamma_1-\beta_1)\Gamma(\beta_2)\Gamma(\gamma_2-\beta_2)}  \label{def:form2}
\\
& \times \int_{0}^{1}  \frac{dt_1 \wedge dt_2}{t_1^{1-\beta_{2}} t_2^{1-\beta_{1}}(1-t_1)^{1+\beta_{2}-\gamma_{2}}(1-t_2)^{1+\beta_{1}-\gamma_{1}}\left(1-x t_2-y t_1\right)^{\alpha}}. \notag
\end{align}}
\\
From \cite{Sasaki1988LinearDE}, we have the following property: 
\begin{prop} \label{quad}
The four linearly independent solutions to the system (\ref{def:F2sys}) are quadratically related if and only if:
\[  
 \alpha=\beta_{1}+\beta_{2}-\frac{1}{2}, \quad \gamma_{1}=2 \beta_{1} , \quad \gamma_{2}=2 \beta_{2}.  
\]
\end{prop}
\vspace{-0.1cm}
We will be interested in the case $\beta_1 = \beta_2 = \frac{1}{2}$ and $\gamma_1 = \gamma_2 = 1$, where the condition is verified. The differential systems (\ref{def:2F1sys}) and (\ref{def:F2sys}) can be written as first-order Pfaffians:
\begin{equation} \label{pfaff} 
    d \vec{f} = A_{_2F_1}(z) \cdot \vec{f}, \qquad  d \vec{F} = A_{F_2}(x,y)\cdot \vec{F}, 
\end{equation} with $
\vec{f} = \left(f(z),z\partial_zf(z)\right)^T$ and $\vec{F} = \left(F(x,y),x\partial_xF(x,y),y\partial_yF(x,y),xy\partial_x\partial_yF(x,y) \right)^T
$. Consider now the tensor product of two copies of the first system:
\[
\tilde{H}=\tilde{f}\bigl(\Lambda_1^{2}\bigr)\otimes \tilde{f}\bigl(\Lambda_2^{2}\bigr),
\]
in some new variables \((\Lambda_1^{2},\Lambda_2^{2})\). Then:
\[
d\tilde{H}
=
\left(
A_{2}F_{1}\bigl(\Lambda_1^{2}\bigr)\otimes I
+
I\otimes A_{2}F_{1}\bigl(\Lambda_2^{2}\bigr)
\right)\tilde{H}.
\] The following proposition shows that this system is related to the Appell $F_2$ system (\ref{def:F2sys}) by a simple gauge transformation \cite[~2.4]{Clingher_2017}: 
 \begin{prop} \label{prop:gauge}
 The connection form of the Appell $F_2$ system verifying the condition of Proposition \ref{quad} satisfies:
\begin{equation} \label{gauge}
    A_{F_2}(x,y)|_{T\left(\Lambda_{1}, \Lambda_{2}\right)} =  g \cdot \left( A_{_2F_1}(\Lambda_1^2) \otimes \mathbb{I}+\mathbb{I} \otimes A_{_2F_1}(\Lambda_2^2)\right) \cdot g^{-1} + dg \cdot g^{-1} ,
\end{equation}
with change of variables $T(\Lambda_1,\Lambda_2)=\left(\frac{4\Lambda_1\Lambda_2}{(\Lambda_1+\Lambda_2)^2},-\frac{(1-\Lambda_1^2)(1-\Lambda_2^2)}{(\Lambda_1+\Lambda_2)^2}\right)$.
 \end{prop} 

The explicit expressions for the matrices can be found in \cite[\S2]{Clingher_2017}, along with duality transformations. For what follows it is only relevant that the first row of the matrix $g$ is given by $(g_{11},0,0,0)$, with $g_{11} = (\Lambda_1 +\Lambda_2 )^{2\beta_1 + 2\beta_2 -1}$.
\section{The conformal two-dimensional traintrack integral}
We consider the family of so-called \emph{fishnet integrals} in two Euclidean dimensions; these were introduced in \cite{fish1} as: \begin{equation} \label{family} \small
I_G(\alpha)  =  
\int  
\prod_{i=1}^{\mathscr{l}} d^2 \xi_i 
 \prod_{i,j} \frac{1}{[(\xi_i-\xi_j)(\xi_i-\alpha_j)(\alpha_i-\alpha_j)]^{1/2}}, \qquad \xi_i,a_i \in \mathbb{R}^2.
\end{equation} 
In \cite{fish1} it was proven that each of these integrals is related to a family $M_G$ of Calabi-Yau $\mathscr{l}$-folds, in particular: \begin{equation} 
I_G(\alpha) = (-i)^{\mathscr{l}}\Pi_G^{\dagger}\Sigma\Pi_G , \label{int:genform}
\end{equation} where $\Sigma$ is the intersection matrix and $\Pi_G$ is the vector of periods of the holomorphic $(\mathscr{l},0)-$ form $\Omega_G$: \begin{equation*} \small
    \Pi_G = \left(\int_{\gamma_1}\Omega_G,\ldots, \int_{\gamma_N}\Omega_G\right)^T, \qquad \gamma_i \in H_{\mathscr{l}}(M_G,\mathbb{Z}).
\end{equation*} 

We are interested in the following element of the family (\ref{family}), called the \emph{conformal traintrack integral}:
\begin{equation*} \small 
    I_G(\alpha) = \int\frac{d\xi_1^2 \ d\xi_2^2}{[(\xi_1-\alpha_1)(\xi_1-\alpha_2)(\xi_1-\alpha_5)(\xi_1-\xi_2)(\xi_2-\alpha_5)(\xi_2-\alpha_3)(\xi_2-\alpha_4)]^{1/2}}.
\end{equation*} 

In \cite{duhr2025feynmanintegralsellipticintegrals} it was shown, after  complexifying the variables and exploiting the Yangian symmetry of the family, that the integral can be written in the form (\ref{int:genform}), with:
\begin{equation} \small
     \Omega_G = \frac{dt_1 \wedge dt_2}{\sqrt{t_1t_2(1-t_1)(1-t_2)(1-z_1t_1-(1-z_2)t_2)}}, \qquad t_i,z_i \in \mathbb{C}. \label{k3diff}
\end{equation}
The integral of this form corresponds, up to constant prefactors, to the integral representation of Appell $F_2$ function (\ref{def:form2}), with $(x,y) \leftrightarrow (z_1,1-z_2)$ and $\beta_1 = \beta_2 = \frac{1}{2}$ and $\gamma_1 = \gamma_2 = 1$. Its periods can then be computed as solutions to the system\footnote{since they are horizontal sections of the Gauss-Manin connection} (\ref{def:F2sys}). In \cite{duhr2025feynmanintegralsellipticintegrals}, a basis of local solutions is computed through the Frobenius method, and the modular parametrization is then obtained from an ansatz, by matching power series expansions. Here we instead prove it using Proposition \ref{prop:gauge}.
\begin{thm}
    A basis of solutions for the Picard-Fuchs system of the holomorphic form (\ref{k3diff}) in a neighborhood of $ (x,y)=(0,0) $, for $|\Lambda_1^2|<1$, $|1-\Lambda_2^2|<1$ is given by:
\begin{equation}
\begin{aligned}
\Pi_0(x, y) &= \frac{4}{\pi^2} (\Lambda_1 +\Lambda_2) K(\Lambda_1^2)\,K(1-\Lambda_2^2),\\
\Pi_1(x, y) &= \frac{4i}{\pi^2}(\Lambda_1 +\Lambda_2) K(1 - \Lambda_1^2)\,K(1-\Lambda_2^2),\\
\Pi_2(x, y) &= \frac{4i}{\pi^2} 
  (\Lambda_1 +\Lambda_2)
  K(\Lambda_1^2)\,K(\Lambda_2^2),\\
\Pi_3(x, y) &= \frac{4i^2}{\pi^2} 
  (\Lambda_1 +\Lambda_2)\,
  K(1 - \Lambda_1^2)\,K(\Lambda_2^2).
\end{aligned}
\label{eq:omega_block}
\end{equation}
 \begin{proof}
  Let $\Psi(z)$ be the period matrix of the $_2F_1$ system (\ref{def:2F1sys}) for $\beta_1,\beta_2 = \frac{1}{2}$, obtained by integrating the vector $ \vec{f} = \left(f(z),z\partial_{z}f(z)\right)^T$ over a canonical basis\footnote{with a symplectic intersection matrix $J =   \left(\begin{smallmatrix}   
       0&1\\-1&0
   \end{smallmatrix}\right)$}. By horizontality of the Gauss-Manin connection, the period matrix satisfies the same Pfaffian system (\ref{pfaff}) as $\vec{f}$. The first row of $\Psi(\Lambda)$ is then made of the periods of the holomorphic form of the \emph{Legendre family}, which can be taken around $z = 0$ as: \[ \label{Kbasis}\frac{2}{\pi}K(z), \quad \frac{2}{\pi}iK(1-z),\]
  where $K(z)$ is the complete elliptic integral of the first kind.
 Let $\Psi(\Lambda_1,\Lambda_2)$ be the period matrix of the tensor product system; by construction its first row is the tensor product of the periods above in the variables $(\Lambda_1^2,\Lambda_2^2)$.  By the gauge transformation (\ref{gauge}), $ \Phi(x,y) = g \cdot\Psi (\Lambda_1,\Lambda_2) $ is then, up to an overall normalization factor, the period matrix for the Appell $F_2$ system. In particular, since in the first row of $g$ the only non-zero element is $g_{11} =(\Lambda_1+ \Lambda_2)$, the result follows.
 \end{proof}
\end{thm}

\begin{rem}
    A similar factorization holds for every configuration of $\beta_1, \beta_2,\gamma_1,\gamma_2$ verifying the conditions of Proposition \ref{quad}, as shown in \cite{Clingher_2017}. In \cite{Griffin_2018}, the cases $\beta_1 = \beta_2 = \tfrac{1}{2}$ and $\beta_1 = \tfrac{1}{6}, \beta_2 = \tfrac{1}{2}$ were shown to be directly related to the \emph{Kummer surface} $ \mathcal{E}_1 \times \mathcal{E}_2$ via an analysis of the corresponding elliptic Jacobian fibrations.
    \end{rem}
\vspace{-0.05cm}
Since the gauge transformation acts on the relevant vector as the scalar $g_{11}$, the Hodge-Riemann bilinear relations follow directly from those of the underlying $_2F_1$ systems:
        \[
        \Pi^T \Sigma \Pi = 0, \quad (-i)^2\Pi^{\dagger} \Sigma \Pi \geq 0,  \qquad\Sigma = J\otimes J =\left(\begin{smallmatrix}   
       0&0&0&1\\0&0&-1&0\\0&-1&0&0\\1&0&0&0
   \end{smallmatrix}\right).\]
These are the quadratic relations expected from Proposition \ref{quad}. In \cite{duhr2025modularformsthreeloopbanana}  it was shown that this form of $\Sigma$ directly implies that $\Pi_0$ and the mirror map can be written in terms of products of modular forms. This follows directly from our basis of solutions. Indeed, we can introduce the canonical coordinates:
\begin{equation*}
    \tau_1 =  \frac{\Pi_1(x,y)}{\Pi_0(x,y)} =i\frac{K(1-\Lambda_1^2)}{K(\Lambda_1^2)},  \qquad \tau_2 = \frac{\Pi_2(x,y)}{\Pi_0(x,y)}  = i\frac{K(\Lambda_2^2)}{K(1-\Lambda_2^2)}.
\end{equation*} Notice that $\tau_1$ depends solely on $\Lambda_1$ and $\tau_2$ solely on $\Lambda_2$, which allows us to directly obtain the modular parameterization using the Hauptmodul $\lambda(\tau)$ of $\Gamma(2)$:
\begin{equation} \label{mirr}
    \Lambda_1^2(\tau_1) = \lambda(\tau_1) = \frac{\theta_2(\tau_1)^4}{\theta_3(\tau_1)^4}, \qquad \Lambda_2^2(\tau_2) = 1-\lambda(\tau_2)= \frac{\theta_4(\tau_2)^4}{\theta_3(\tau_2)^4},
\end{equation} where the \emph{Jacobi theta functions}  $\theta_i(\tau)$ are defined as: \[ \label{def:theta} 
        \theta_2(\tau) :=\sum_{n \in \mathbb{Z}}q^{(n+\frac{1}{2})^2}, \quad 
        \theta_3(\tau) :=\sum_{n \in \mathbb{Z}}q^{n^2}, \quad
        \theta_4(\tau) :=\sum_{n \in \mathbb{Z}}(-1)^nq^{n^2}.
    \]These are modular forms of weight $\frac{1}{2}$ for $\Gamma(2)$. Using $K(\lambda(\tau_i)) = \frac{\pi}{2}\theta_3(\tau_i)^2$, we can express the holomorphic period in modular coordinates: 
    \begin{equation*}
    \Pi_0(\tau_1,\tau_2) = \theta_2(\tau_1)^2\theta_3(\tau_2)^2 +\theta_4(\tau_2)^2\theta_3(\tau_1)^2,
\end{equation*} so we find that it is a modular form of weight $(1,1)$ and trivial character with respect to the monodromy group $\Gamma(2) \times \Gamma(2)$, according to the definitions in \cite[Def.~1.1]{yang}. The mirror map (\ref{mirr}) is instead a modular function. This is in agreement with the results of \cite{Griffin_2018} and \cite{duhr2025modularformsthreeloopbanana}.

\maketitle


\end{document}